\documentclass[11pt,a4paper,british]{article}
\usepackage{yfonts, amsmath, amscd, amsthm, latexsym, anysize, graphicx, makeidx, appendix, type1cm, eso-pic, color, wasysym, float, bbold, dsfont, xcolor, framed, mdframed, pzccal, anyfontsize, stmaryrd, euscript} 
\usepackage[colorlinks, pagebackref, pdfusetitle, urlcolor=black, citecolor=black, linkcolor=black, bookmarksnumbered, plainpages=false]{hyperref}
\usepackage{yfonts, amsmath, amscd, amsthm, amssymb, latexsym, anysize, graphicx, makeidx, appendix, type1cm, eso-pic, color, wasysym, float, xcolor, framed, mdframed, pzccal, anyfontsize, courier, babel, titlesec, dsfont}
\usepackage [all,pdf,cmtip]{xy}
\SetSymbolFont{largesymbols}{normal}{OMX}{iwona}{m}{n}
\usepackage[T1]{fontenc}

\usepackage{mathpazo} 
\makeindex
\marginsize{2.5cm}{2.5cm}{2cm}{2cm}
\graphicspath{{/Users/Ramses/Documents/Universitat/02_Thesis/Pictures/}} 

\makeatletter
\def\th@plain{\thm@notefont{} \itshape}
\def\th@definition{\thm@notefont{}\normalfont}
\makeatother

\DeclareFontFamily{U}{futm}{}
\DeclareFontShape{U}{futm}{m}{n}{
  <-> s * [1] fourier-bb
  }{}
\DeclareSymbolFont{Ufutm}{U}{futm}{m}{n}
\DeclareSymbolFontAlphabet{\mathbb}{Ufutm}
\mathchardef\mhyphen="2D 

\makeatletter 
\DeclareRobustCommand\widecheck[1]{{\mathpalette\@widecheck{#1}}}
\def\@widecheck#1#2{%
    \setbox\z@\hbox{\m@th$#1#2$}%
    \setbox\tw@\hbox{\m@th$#1%
       \widehat{%
          \vrule\@width\z@\@height\ht\z@
          \vrule\@height\z@\@width\wd\z@}$}%
    \dp\tw@-\ht\z@
    \@tempdima\ht\z@ \advance\@tempdima2\ht\tw@ \divide\@tempdima\thr@@
    \setbox\tw@\hbox{%
       \raise\@tempdima\hbox{\scalebox{1}[-1]{\lower\@tempdima\box
\tw@}}}%
    {\ooalign{\box\tw@ \cr \box\z@}}}
\makeatother

\renewcommand{\thefootnote}{\fnsymbol{footnote}}

\newtheorem{defn0}{Definition}[section]
\newtheorem{prop0}[defn0]{Proposition}
\newtheorem{thm0}[defn0]{Theorem}
\newtheorem{lemma0}[defn0]{Lemma}
\newtheorem{corollary0}[defn0]{Corollary}
\newtheorem{example0}[defn0]{Example}
\newtheorem{remark0}[defn0]{Remark}
\newtheorem{conjecture0}[defn0]{Conjecture}

\newenvironment{proposition}{\smallskip\begin{prop0}}{\end{prop0}}
\newenvironment{lemma}{\smallskip\begin{lemma0}}{\end{lemma0}}
\newenvironment{example}{\smallskip\begin{example0}}{\end{example0}}
\newenvironment{remark}{\smallskip\begin{remark0}}{\end{remark0}}

\newcommand{\ob}{\operatorname{Obj}}

\newcommand{\ibimod}{\operatorname{\mathpzc{iBimod}}}
\newcommand{\bimod}{\operatorname{\mathpzc{Bimod}}}
\newcommand{\ivect}{\operatorname{\mathpzc{iVect}}}

\newcommand{\homo}{\operatorname{Hom}}
\newcommand{\coder}{\operatorname{Coder}}
\newcommand{\icoder}{\operatorname{iCoder}}

\newcommand{\id}{\operatorname{Id}}

\newcommand{\K}{\operatorname{\mathbb{K}}}
\newcommand{\Z}{\operatorname{\mathbb{Z}}}

\newcommand{\h}{\operatorname{H}}

\newcommand{\ba}{\operatorname{Bar}}

\renewcommand{\thefootnote}{\fnsymbol{footnote}}
\newcommand{\hh}{\operatorname{HH}}

\begin{document}

\title{Hochschild homology and cohomology for involutive $A_\infty$-algebras}
\author{Rams\`es Fern\`andez-Val\`encia}
\date{}

\maketitle


\begin{abstract}
\noindent
We present a study of the homological algebra of bimodules over $A_\infty$-algebras endowed with an involution. Furthermore we introduce a derived description of Hochschild homology and cohomology for involutive $A_\infty$-algebras.
\end{abstract}

\tableofcontents

\let\thefootnote\relax\footnotetext{\date{\today}}
\let\thefootnote\relax\footnotetext{Email: ramses.fernandez.valencia@gmail.com}

\section{Introduction}

Hochschild homology and cohomology are homology and cohomology theories developed for associative algebras which appears naturally when one studies its deformation theory. Furthermore, Hochschild homology plays a central role in topological field theory in order to describe the closed states part of a topological field theory.
\\\\
An involutive version of Hochschild homology and cohomology was developed by Braun in \cite{Braun13} by considering associative and $A_\infty$-algebras endowed with an involution and morphisms which commute with the involution.
\\\\
This paper pretends to take a step further with regards to \cite{FeGi15}. Whilst in the latter paper we develop the homological algebra required to give a derived version of Braun's involutive Hochschild homology and cohomology for involutive associative algebras, this research is devoted to develop the machinery required to give a derived description of involutive Hochschild homology and cohomology for $A_\infty$-algebras endowed with an involution.
\\\\
As in \cite{FeGi15}, this research has been driven by the author's research on Costello's classification of topological conformal field theories \cite{Costello07}, where he proves that an open 2-dimensional theory is equivalent to a Calabi-Yau $A_\infty$-category. In \cite{Ramses15}, the author extends the picture to unoriented topological conformal field theories, where open theories now correspond to involutive Calabi-Yau $A_\infty$-categories, and the closed state space of the universal open-closed extension turns out to be the involutive Hochschild chain complex of the open state algebra.

\section{Basic concepts}

\subsection{Coalgebras and bicomodules} 

An \index{Involutive graded coalgebra} \textit{involutive graded coalgebra} over a field $\K$ is a graded $\K$-module $C$ endowed with a coproduct $\Delta:C\to C\otimes_{\K} C$ of degree zero together with an involution $\star:C\to C$ such that:
\begin{enumerate}
\item The map $\Delta$ makes the following diagram commute
\begin{equation}
\xymatrix{
C\ar[rr]^-{\Delta}\ar[d]_-{\Delta} & & C\otimes_{\K} C \ar[d]^-{\Delta\otimes\id_C} \nonumber \\ 
C\otimes_{\K} C\ar[rr]_-{\id_C\otimes \Delta} & & C\otimes_{\K} C\otimes_{\K} C
}
\end{equation}
\item the involution and $\Delta$ are compatible: $\Delta(c^\star)=(\Delta(c))^\star$, for $c\in C$, where the involution on $C\otimes_{\K} C$ is defined as: $(c_1\otimes c_2)^\star=c_2^\star\otimes c_1^\star$, for $c_1,c_2\in C$.
\end{enumerate}

An \index{Involutive coderivation} \textit{involutive coderivation} on an involutive coalgebra $C$ is a map $L:C\to C$ preserving involutions and making the following diagram commutative:
\begin{equation}
\xymatrixcolsep{3pc}\xymatrix{
C\ar[rr]^-{L}\ar[d]_-{\Delta} & & C \ar[d]^-{\Delta} \nonumber \\
C\otimes_{\K} C\ar[rr]_-{L\otimes \id_C+\id_C\otimes L} & & C\otimes_{\K} C
}
\end{equation}
Denote with $\icoder(-)$ the spaces of coderivations of involutive coalgebras. Observe that $\icoder(-)$ are Lie subalgebras.
\\\\
An \index{Involutive DG-coalgebra} \textit{involutive differential graded coalgebra} is an involutive coalgebra $C$ equipped with an involutive coderivation $b:C\to C$ of degree $-1$ such that $b^2=b\circ b=0$.
\newpage
A morphism between two involutive coalgebras $C$ and $D$ is a graded map $C\stackrel{f}{\longrightarrow} D$ compatible with the involutions which makes the following diagram commutative:
\begin{equation} \label{eq:coalgebra3}
\xymatrix{
C\ar[rr]^-{f}\ar[d]_-{\Delta_C} & & D \ar[d]^-{\Delta_D} \\
C\otimes_{\K} C\ar[rr]_-{f\otimes f} & & D\otimes_{\K} D
}
\end{equation}

\begin{example}
The cotensor coalgebra of an involutive graded $\K$-bimodule $A$ is defined as $\widehat{T}A=\bigoplus_{n\geq 0}A^{\otimes_{\K} n}$. We define an involution in $A^{\otimes_{\K} n}$ by stating: 
$$(a_1\otimes \dots \otimes a_n)^\star:=(a_n^\star \otimes \dots \otimes a_1^\star).$$ 
The coproduct on $\widehat{T}A$ is given by:
$$\Delta(a_1\otimes\cdots \otimes a_n)=\sum_{i=0}^n(a_1\otimes\cdots\otimes a_i)\otimes(a_{i+1}\otimes\cdots\otimes a_n).$$
Observe that $\Delta$ commutes with the involution.
\end{example}


\begin{proposition}\label{alterinfi}
There is a canonical isomorphism of complexes:
\begin{displaymath}
\icoder\left(TSA\right)\cong \homo_{A\mhyphen\ibimod}\left(\ba(A),A\right).
\end{displaymath}
\end{proposition}

\begin{proof}
The proof follows the arguments in Proposition 4.1.1 \cite{FeGi15}, where we show the result for the non-involutive setting in order to restrict to the involutive one.
\\\\
Since $\ba(A)=A\otimes_{\K}TSA\otimes_{\K}A$, the degree $-n$ part of $\homo_{A\mhyphen\bimod}(\ba(A),A)$ is
the space of degree $-n$ linear maps $TSA\to A$, which is isomorphic to the space of degree $(-n-1)$ linear maps $TSA \to SA$. By the universal property of the tensor coalgebra, there is a bijection between degree $(-n-1)$ linear maps $TSA \to SA$ and degree $(-n-1)$ coderivations on $TSA$. Hence the degree $n$ part of $\homo_{A\mhyphen\bimod}(\ba(A),A)$ is isomorphic to the degree $n$ part of $\coder\left(TSA\right)$. One checks directly that this isomorphism restricts to an isomorphism of graded vector spaces
$$\homo_{A\mhyphen\ibimod}(\ba(A),A)\cong\icoder\left(TSA\right).$$
Finally, one can check that the differentials coincide under the above isomorphism, cf. Section 12.2.4 \cite{LoVa12}.
\end{proof}

\begin{remark}
Proposition \ref{alterinfi} allows us to think of a coderivation as a map $\widehat{T}A\to A$. Such a map $f:\widehat{T}A\to A$ can be described as a collection of maps $\{f_n:A^{\otimes n}\to A\}$ which will be called the components of $f$.
\end{remark}

If $b$ is a coderivation of degree $-1$ on $\widehat{T}A$ with $b_n:A^{\otimes_{\K} n}\to A$, then $b^2$ becomes a linear map of degree $-2$ with
\begin{displaymath}
b^2_n=\sum_{i+j=n+1}\sum_{k=0}^{n-1}b_i\circ\left(\id^{\otimes k}\circ b_j\circ \id^{\otimes(n-k-j)}\right).
\end{displaymath} 

The coderivation $b$ will be a differential for $\widehat{T}A$ if, and only if, all the components $b^2_n$ vanish.
\\\\
Given a (involutive) graded $\K$-bimodule $A$, we denote the suspension of $A$ by $SA$ and define it as the graded (involutive) $\K$-bimodule with $SA_i=A_{i-1}$. Given such a bimodule $A$, we define the following morphism of degree $-1$ induced by the identity $s:A\to SA$ by $s(a)=a$.

\begin{lemma}[cf. Lemma 1.3 \cite{GeJo90}]
If $b_k:(SA)^{\otimes_{\K} k}\to SA$ is an involutive linear map of degree $-1$, we define $m_k:A^{\otimes_{\K} k}\to A$ as $m_k=s^{-1}\circ b_k\circ s^{\otimes_{\K} k}$.
Under these conditions:
\begin{displaymath}
b_k(sa_1\otimes \dots \otimes sa_k)=\sigma m_k(a_1\otimes\dots\otimes a_k),
\end{displaymath}
where $\sigma:=(-1)^{(k-1)|a_1|+(k-2)|a_2|+\cdots +2|a_{k-2}|+|a_{k-1}|+\frac{k(k-1)}{2}}$.
\end{lemma}

\begin{proof}
The proof follows the arguments of Lemma 1.3 \cite{GeJo90}. We only need to observe that the involutions are preserved as all the maps involved in the proof are assumed to be involutive.
\end{proof}

Let $\overline m_k:=\sigma m_k$, then we have $b_k(sa_1 \otimes \dots \otimes sa_k)=\overline m_k(a_1\otimes\dots \otimes a_k)$.

\begin{proposition}
Given an involutive graded $\K$-bimodule $A$, let $\epsilon_i=|a_1|+\dots+|a_i|-i$ for $a_i\in A$ and $1\leq i\leq n$. A boundary map $b$ on $\widehat{T}SA$ is given in terms of the maps $\overline m_k$ by the following formula:
\begin{eqnarray}
& & b_n(sa_1\otimes \dots \otimes sa_n) \nonumber \\ 
& = & \sum_{k=0}^n\sum_{i=1}^{n-k+1}(-1)^{\epsilon_{i-1}}(sa_1\otimes \dots \otimes sa_{i-1}\otimes \overline m_k(a_i\otimes \dots \otimes a_{i+k-1})\otimes \dots \otimes sa_n) \nonumber.
\end{eqnarray}
\end{proposition}

\begin{proof}
This proof follows the arguments of Proposition 1.4 \cite{GeJo90}. The only detail that must be checked is that $b_n$ preserves involutions:
\begin{align*}
 & b_n((sa_1\otimes \dots \otimes sa_n)^\star) \\
 & =  \sum_{j,k}\pm(sa_n^\star\otimes \dots \otimes sa_j^\star \otimes \overline{m}_k(a^\star_{j-1} \otimes \dots \otimes a^\star_{j-k+1})\otimes \dots \otimes sa^\star_1) \\
& =  \sum_{j,k}\pm(sa_1\otimes \dots \otimes \overline{m}_k(a_{j-k+1}\otimes\dots\otimes a_{j-1})\otimes sa_j\otimes \dots \otimes sa_n)^\star \\
& =  (b_n(sa_1\otimes \dots \otimes sa_n))^\star. \qedhere
\end{align*}
\end{proof}



Given an involutive coalgebra $C$ with coproduct $\rho$ and counit $\varepsilon$, for an involutive graded vector space $P$, a \index{Left coaction} \textit{left coaction} is a linear map $\Delta^L: P\to C\otimes_{\K} P$ such that
\begin{enumerate}
\item $(\id\otimes \rho)\circ \Delta^L=(\rho\otimes \id)\circ \Delta^L$;
\item $(\id\otimes\varepsilon)\circ=\id$.
\end{enumerate}
Analagously we introduce the concept of \index{Right coaction} \textit{right coaction}.
\\\\
Given an involutive coalgebra $(C,\rho,\varepsilon)$ with involution $\star$ we define an \index{Involutive $C$-bicomodule} \textit{involutive $C$-bicomodule} as an involutive graded vector space $P$ with involution $\dagger$, a left coaction $\Delta^L:P\to C\otimes_{\K} P$ and a right coaction $\Delta^R:P\to P\otimes_{\K} C$ which are compatible with the involutions, that is the diagrams below commute:

\begin{minipage}{.5\textwidth}
\begin{equation} \label{bicomodule3}
\xymatrix{
P\ar[rr]^-{(-)^\star}\ar[d]_-{\Delta^L} & & P \ar[d]^-{\Delta^R} \\
C\otimes_{\K} P \ar[rr]_-{(-,-)^\star} & & P\otimes_{\K} C 
}
\end{equation}
\end{minipage}%
\begin{minipage}{.5\textwidth}
\begin{equation} \label{bicomodule2}
\xymatrix{
P\ar[rr]^-{\Delta^L}\ar[d]_-{\Delta^R} & & C\otimes P\ar[d]^-{\id_C\otimes \Delta^R} \\
P\otimes_{\K} C\ar[rr]_-{\Delta^L\otimes \id_C} & & C\otimes_{\K} P \otimes_{\K} C
}
\end{equation}
\end{minipage}

Where
\[
\left.\begin{array}{cccc}(-,-)^\star: & C\otimes_{\K} P & \to & P\otimes_{\K} C \\ & c\otimes p & \mapsto & p^{\dagger}\otimes c^{\star}\end{array}\right.
\]
For two involutive $C$-bicomodules $(P_1,\Delta_1)$ and $(P_2,\Delta_2)$, a morphism $P_1\stackrel{f}{\longrightarrow} P_2$ is defined as an involutive morphism making diagrams below commute:

\begin{minipage}{.5\textwidth}
\begin{equation} 
\xymatrix{
P_1 \ar[r]^-{\Delta^L_1}\ar[d]_-f & C\otimes_{\K} P_1\ar[d]^-{\id_C\otimes f}\\
P_2 \ar[r]_-{\Delta_2^L} & C\otimes_{\K} P_2
}
\end{equation}
\end{minipage}%
\begin{minipage}{.5\textwidth}
\begin{equation}
\xymatrix{
P_1 \ar[r]^-{\Delta^R_1}\ar[d]_-f & P_1\otimes_{\K} C\ar[d]^-{\id_C\otimes f}\\
P_2 \ar[r]_-{\Delta_2^R} & P_2\otimes_{\K} C
}
\end{equation}
\end{minipage}

\subsection{$A_\infty$-algebras and $A_\infty$-quasi-isomorphisms} \label{infinityalgebras} 

An \index{Involutive $\K$-algebra} \textit{involutive $\K$-algebra} is an algebra $A$ over a field $\K$ endowed with a $\K$-linear map (an involution) $\star:A\to A$ satisfying:
\begin{enumerate}
  \item $0^\star=0$ and $1^\star=1$;
  \item $(a^\star)^\star=a$ for each $a\in A$;
  \item $(a_1 a_2)^\star=a_2^\star a_1^\star$ for every $a_1,a_2\in A$.
\end{enumerate}

\begin{example}
\begin{enumerate}
\item Any commutative algebra $A$ becomes an involutive algebra if we endoew it with the identity as involution.
\item Let $V$ an involutive vector space. The tensor algebra $\bigoplus_n V^{\otimes n}$ becomes an involutive algebra if we endow it with the following involution: $(v_1,\dots, v_n)^\star=(v_n^\star,\dots,v_1^\star)$. This example is particularly important and we will come back to it later on.
\item For a discrete group $G$, the group ring $\K[G]$ is an involutive $\K$-algebra with involution given by inversion $g^\star=g^{-1}$.
\end{enumerate}
\end{example}

Given an involutive algebra $A$, an \index{Involutive $A$-bimodule} \textit{involutive $A$-bimodule} $M$ is an $A$-bimodule endowed with an involution satisfying $(a_1 m a_2)^\star=a_2^\star m^\star a_1^\star.$
\\\\
Given two involutive $A$-bimodules $M$ and $N$, a \index{Involutive morphism} \textit{involutive morphism} between them is a morphism of $A$-bimodules 
$f:M\to N$ 
compatible with the involutions.
\begin{lemma}
The composition of involutive morphisms is an involutive morphism.
\end{lemma}

\begin{proof}
Given $f:M\to N$ and $g:N\to P$ two involutive morphisms: 
\begin{displaymath}
(f\circ g)(m^\star)=f((g(m^\star)))=f((g(m))^\star)=(f(g(m)))^\star \qedhere
\end{displaymath}
\end{proof}
Involutive $A$-bimodules and involutive morphisms form the category $A\mhyphen\ibimod$.
\\\\
Given a (involutive) graded $\K$-module $A$, we denote the suspension of $A$ by $SA$ and define it as the graded (involutive) $\K$-module with $SA_i=A_{i-1}$. An \index{Involutive $A_\infty$-algebra} \textit{involutive $A_\infty$-algebra} is an involutive graded vector space $A$ endowed with involutive morphisms
\begin{equation}\label{infinity6}
b_n:(SA)^{\otimes_{\K} n}\to SA,\, n\geq 1,
\end{equation}
of degree $n-2$ such that the identity below holds:
\begin{equation} \label{alginfinity2}
\sum_{i+j+l=n}(-1)^{i+jl}b_{i+1+l}\circ(\id^{\otimes i}\otimes b_j\otimes \id^{\otimes l})=0,\, \forall n\geq 1. 
\end{equation}

\begin{remark}
Condition (\ref{alginfinity2}) says, in particular, that $b_1^2=0$.
\end{remark}

\begin{example}
\begin{enumerate}
\item The concept $A_\infty$-algebra is a generalization for that of a differential graded algebra. Indeed, if the maps $b_n=0$ for $n\geq 3$ then $A$ is a differential $\mathbb{Z}$-graded algebra and conversely an $A_\infty$-algebra $A$ yields a differential graded algebra if we require $b_n=0$ for $n\geq 3$.
\item The definition of $A_\infty$-algebra was introduced by Stasheff whose motivation was the study of the graded abelian group of singular chains on the based loop space of a topological space.
\end{enumerate}
\end{example}

For an involutive $A_\infty$-algebra $(A,b_n)$, the \index{Involutive bar complex} \textit{involutive bar complex} is the involutive differential graded coalgebra $\ba(A)=\widehat{T}SA$, where we endow $\ba(A)$ with a coderivation defined by $b_i=s^{-1}\circ b_i\circ s^{\otimes_Ki}$ (cf. Definition 1.2.2.3 \cite{Lefevre03}).
\\\\
Given two involutive $A_\infty$-algebras $C$ and $D$, a \index{Morphism of $A_\infty$-algebras} \textit{{morphism of $A_\infty$-algebras}} $f:C\to D$ is an involutive morphism of degree 0 between the associated involutive differential graded coalgebras $\ba(C)\to \ba(D)$. 
\\\\
It follows from Proposition \ref{alterinfi} that the definition of an involutive $A_\infty$-algebra can be summarized by saying that it is an involutive graded $\K$-module $A$ equipped with an involutive coderivation on $\ba(A)$ of degree $-1$.

\begin{remark}
From \cite{Braun13}, Definition 2.8, we have that a morphism of involutive $A_\infty$-algebras $f:C\to D$ can be given by an involutive morphism of differential graded coalgebras $\ba(C)\to \ba(D)$, that is, a series of involutive homogeneous maps of degree zero
\begin{displaymath}
f_n:(SC)^{\otimes_{\K} n}\to SD,\, n\geq 1,
\end{displaymath} 
such that 
\begin{equation} \label{infinity4}
\sum_{i+j+l=n}f_{i+l+1}\circ\left(\id_{SC}^{\otimes i}\otimes\,b_j\otimes \id_{SC}^{\otimes l}\right)=\sum_{i_1+\dots +i_s=n}b_s\circ (f_{i_1}\otimes\cdots\otimes f_{i_s}).
\end{equation}
The composition $f\circ g$ of two morphisms of involutive $A_\infty$-algebras is given by
\begin{equation}
(f\circ g)_n=\sum_{i_1+\dots+i_s=n}f_s\circ(g_{i_1}\otimes\cdots\otimes g_{i_s}); \nonumber
\end{equation}
\\
the identity on $SC$ is defined as $f_1=\id_{SC}$ and $f_n=0$ for $n\geq 2$.
\end{remark}




For an involutive $A_\infty$-algebra $A$, we define its associated homology algebra $\h_\bullet(A)$ as the homology of the differential $b_1$ on $A$: $\h_\bullet(A)=H_\bullet(A,b_1).$

\begin{remark}
Endowed with $b_2$ as multiplication, the homology of an $A_\infty$-algebra $A$ is an associative graded algebra, whereas $A$ is not usually associative.
\end{remark}

Let $f:A_1\to A_2$ a morphism of involutive $A_\infty$-algebras with components $f_n$; for $n=1$, $f_1$ induces a morphism of algebras $\h_{\bullet} (A_1)\to \h_{\bullet}(A_2)$. We say that $f:A_1\to A_2$ is an \index{$A_\infty$-quasi-isomorphism} \textit{$A_\infty$-quasi-isomorphism} if $f_1$ is a quasi-isomorphism.


\subsection{$A_\infty$-bimodules} \label{infbim} 

Let $(A, b^A)$ be an involutive $A_\infty$-algebra. An \index{Involutive $A_\infty$-bimodule} \textit{involutive $A_\infty$-bimodule} is a pair $(M, b^M)$ where $M$ is a graded involutive $\K$-module and $b^M$ is an involutive differential on the $\ba(A)$-bicomodule 
$$\ba(M):=\ba(A)\otimes_{\K} SM \otimes_{\K} \ba(A).$$
Let $\left(M,b^M\right)$ and $\left(N, b^N\right)$ be two involutive $A_\infty$-bimodules. We define a \index{Morphism of inv. $A_\infty$-bimodules} \textit{morphism of involutive $A_\infty$-bimodules} $f:M\to N$ as a morphism of $\ba(A)$-bicomodules 
$$F:\ba(M)\to\ba(N)$$ 
such that $b^N\circ F=F\circ b^M$.



\begin{proposition}
If $f:A_1\to A_2$ is a morphism of involutive $A_\infty$-algebras, then $A_2$ becomes an involutive bimodule over $A_1$.
\end{proposition}

\begin{proof}
As we are assuming that both $A_1$ and $A_2$ are involutive $A_\infty$-algerbas and that $f$ is involutive, we do not need to care about involutions. When it comes to the bimodule structure, this result holds as $\ba(A_2)$ is made into a bicomodule of $\ba(A_1)$ by the homomorphism of involutive coalgebras $f:\ba(A_1)\to\ba(A_2)$, see Proposition 3.4 \cite{GeJo90}.
\end{proof}

\begin{remark}[Section 5.1 \cite{KoSo09}]
Let $\ivect$ be the category of involutive $\Z$-graded vector spaces and involutive morphisms. For an involutive $A_\infty$-algebra $A$, involutive $A$-bimodules and their respective morphisms form a differential graded category. 
Indeed, following \cite{KoSo09}, Definition 5.1.5: let $A$ be an involutive $A_\infty$-algebra and let us define the category $\overline{A\mhyphen\ibimod}$ whose class of objects are involutive $A$-bimodules and where $\homo_{\overline{A\mhyphen\ibimod}}(M,N)$ is: 
$$\underline{\homo}_{\ivect}^n(\ba(A)\otimes_{\K} SM\otimes_{\K} \ba(A),\ba(A)\otimes_{\K} SN\otimes_{\K} \ba(A)).$$ 
Let us recall that 
$$\underline{\homo}_{\ivect}^n(\ba(A)\otimes_{\K} SM\otimes_{\K} \ba(A),\ba(A)\otimes_{\K} SN\otimes_{\K} \ba(A))$$
is by definition 
$$\prod_{i\in\Z}\homo_{\ivect}((\ba(A)\otimes_{\K} SM\otimes_{\K} \ba(A))^i,(\ba(A)\otimes_{\K} SN\otimes_{\K} \ba(A))^{i+n}).$$ 
The morphism
\begin{multline*}
\underline{\homo}_{\ivect}^n(\ba(A)\otimes_{\K} SM\otimes_{\K} \ba(A),\ba(A)\otimes_{\K} SN\otimes_{\K} \ba(A)) \to \\
\underline{\homo}_{\ivect}^{n+1}(\ba(A)\otimes_{\K} SM\otimes_{\K} \ba(A),\ba(A)\otimes_{\K} SN\otimes_{\K} \ba(A))
\end{multline*}

sends a family $\{f_i\}_{i\in\Z}$ to a family $\{b^N\circ f_i-(-1)^nf_{i+1}\circ b^M\}_{i\in\Z}$. Observe that the zero cycles in 
$\underline{\homo}_{\ivect}^\bullet(\ba(A)\otimes_{\K} M\otimes_{\K} \ba(A), \ba(A)\otimes_{\K} N\otimes_{\K} \ba(A))$
are precisely the morphisms of involutive $A$-bimodules. This morphism defines a differential, indeed: for fixed indices $i,n\in \Z$ we have
\begin{eqnarray}
d^2(f_i) & = & d\left(b^Nf_i-(-1)^nf_{i+1}b^M\right) \nonumber \\
 & = & b^N\left(b^Nf_i-(-1)^nf_{i+1}b^M\right)-(-1)^{n+1}\left(b^Nf_i-(-1)^nf_{i+1}b^M\right)b^M \nonumber \\
 & \stackrel{(!)}{=} & -(-1)^nb^Nf_{i+1}b^M-(-1)^{n+1}b^Nf_{i+1}b^M=0, \nonumber
\end{eqnarray}
where (!) points out the fact that $b^N\circ b^N=0=b^M\circ b^M$.
\\\\
For a morphism $\phi\in\underline{\homo}_{\ivect}^n(\ba(A)\otimes_{\K} M\otimes_{\K} \ba(A),\ba(A)\otimes_{\K} N\otimes_{\K} \ba(A))$
and an element $x\in \ba(A)\otimes_{\K} M\otimes_{\K} \ba(A)$, $\homo_{\overline{A\mhyphen\ibimod}}(M,N)$ becomes an involutive complex if we endowed it with the involution $\phi^\star(x)=\phi(x^\star)$.
\end{remark}

The functor $\homo_{\overline{A\mhyphen\ibimod}}(M,-)$ pairs an involutive $A$-bimodule $F$ with the involutive $\K$-vector space $\homo_{\overline{A\mhyphen\ibimod}}(M,F)$ of involutive homomorphisms. Given a homomorphism $f:F\to G$, for $F,G\in \ob\left(\overline{A\mhyphen\ibimod}\right)$, $\homo_{\overline{A\mhyphen\ibimod}}(M,-)$ pairs $f$ with the involutive map: 
\begin{displaymath}
\left.\begin{array}{cccc}f_\star: & \homo_{\overline{A\mhyphen\ibimod}}(M,F) & \rightarrow & \homo_{\overline{A\mhyphen\ibimod}}(M,G) \\ & \phi & \mapsto & f\circ\phi \end{array}\right..
\end{displaymath}

We prove that $f_\star$ preserves involutions:
\begin{displaymath}
(f_\star\phi^\star)(x)=(f\circ \phi^\star)(x)=f(\phi(x^\star))=f((\phi(x))^\star)=(f(\phi(x)))^\star=(f_\star\phi(x))^\star.
\end{displaymath}

We define the functor $\homo_{\overline{A\mhyphen\ibimod}}(-,M)$, which sends an involutive homomorphism $f:F\to G$, for $F,G\in \ob\left(\overline{A\mhyphen\ibimod}\right)$, to 
\begin{displaymath}
\left.\begin{array}{cccc}\varphi: & \homo_{\overline{A\mhyphen\ibimod}}(G,M) & \rightarrow & \homo_{\overline{A\mhyphen\ibimod}}(F,M) \\ & \phi & \mapsto & \phi\circ f \end{array}\right.
\end{displaymath}
Let us check that the involution is preserved:
\begin{displaymath}
\varphi(\phi^\star)(x)=(\phi^\star\circ f)(x)=\phi(f(x)^\star)=\phi(f(x^\star))=\varphi(\phi)(x^\star)=(\varphi(\phi))^\star(x)
\end{displaymath}

Let $A$ be an involutive $A_\infty$-algebra and let $\left(M, b^M\right)$ and $\left(N, b^N\right)$ be involutive $A$-bimodules. For $f,g:M\to N$  morphisms of $A$-bimodules,
an \index{$A_\infty$-homotopy} \textit{$A_\infty$-homotopy} between $f$ and $g$ is a morphism $h:M\to N$ of $A$-bimodules 
satisfying 
$$f-g=b^N\circ h+h\circ b^M.$$ 
We say that two morphisms $u:M\to N$ and $v:N\to M$ of involutive $A$-bimodules are \index{Homotopy equivalence} \textit{homotopy equivalent} if $u\circ v\sim \id_N$ and $v\circ u\sim \id_M$.

\section{The involutive tensor product} 

For an involutive $A_\infty$-algebra $A$ 
and involutive $A$-bimodules $M$ and $N$, the involutive tensor product $M\widetilde{\boxtimes}_\infty N$ is the following object in $\ivect_{\K}$:
$$M\widetilde{\boxtimes}_\infty N:=\frac{M\otimes_{\K} \ba(A)\otimes_{\K} N}{(m^\star\otimes a_1\otimes\dots \otimes a_k\otimes n-m\otimes a_1\otimes\dots \otimes a_k\otimes n^\star)}.$$
Observe that, for an element of $M\widetilde{\boxtimes}_\infty N$ of the form $m\otimes a_1\otimes\dots\otimes a_k\otimes n$, we have: $(m\otimes a_1\otimes\dots\otimes a_k\otimes n)^\star=m^\star\otimes a_1\otimes\dots\otimes a_k\otimes n=m\otimes a_1\otimes\dots\otimes a_k\otimes n^\star$.



\begin{proposition}
For an involutive $A_\infty$-algebra $A$ and involutive $A$-bimodules $M,N$ and $L$, the following isomorphism holds:
$$\tau:\homo_{\ivect}\left(M\widetilde{\boxtimes}_\infty N, L\right)\stackrel{\cong}{\longrightarrow}\homo_{\ivect}\left(\frac{M\otimes_{\K} \ba(A)}{\sim}, \homo_{\overline{A\mhyphen\ibimod}}(N,L)\right),$$
where in $M\otimes_{\K} \ba(A):\,(m\otimes a_1\otimes \dots\otimes a_k)^\star=m^\star\otimes a_1\otimes \dots\otimes a_k$, $\sim$ denotes the relation $m\otimes a_1\otimes \dots\otimes a_k= m^\star\otimes a_1\otimes \dots\otimes a_k$ and $\frac{M\otimes_{\K} \ba(A)}{\sim}$ has the identity map as involution.
\end{proposition}

\begin{proof}
Let $f:M\widetilde{\boxtimes}_\infty N\to L$ be an involutive map. We define:
$$\tau(f):=\tau_f\in\homo_{\ivect}\left(\frac{M\otimes_{\K}  \ba(A)}{\sim}, \homo_{\overline{A\mhyphen\ibimod}}(N,L)\right),$$
where $\tau_f(m\otimes a_1\otimes \dots \otimes a_k):=\tau_f[m\otimes a_1\otimes \dots \otimes a_k]\in \homo_{\overline{A\mhyphen\ibimod}}(N,L)$. Finally, for $n\in N$ we define: 
$$\tau_f[m\otimes a_1\otimes \dots \otimes a_k](n):=f\left(m\otimes a_1\otimes\dots\otimes a_k\otimes n\right).$$ 
We need to check that $\tau$ preserves the involutions, indeed:
\begin{multline*}
$$\tau_{f^\star}[m\otimes a_1\otimes \dots \otimes a_k](n)=f^\star\left(m\otimes a_1\otimes\dots\otimes a_k\otimes n\right)=\\
=\left(f\left(m\otimes a_1\otimes\dots\otimes a_k\otimes n\right)\right)^\star=(\tau_f)^\star[m\otimes a_1\otimes \dots \otimes a_k](n).$$
\end{multline*}

In order to see that $\tau$ is an isomorphism, we build an inverse. Let us consider an involutive map
\[
\left.\begin{array}{cccc}g_1: & \frac{M\otimes_{\K}  \ba(A)}{\sim} & \to & \homo_{\overline{A\mhyphen\ibimod}}(N,L) \\ & m\otimes a_1\otimes\dots\otimes a_k & \mapsto & g_1[m\otimes a_1\otimes\dots\otimes a_k]\end{array}\right.
\]
 
and define a map
\[
\left.\begin{array}{cccc}g_2: & M\widetilde{\boxtimes}_\infty N & \to & L \\ & m\otimes a_1\otimes\dots\otimes a_k\otimes n & \mapsto & g_1[m\otimes a_1\otimes\dots\otimes a_k](n)\end{array}\right.
\]
We check that $g_2$ is involutive:
\begin{multline}
  g_2((m\otimes a_1\otimes \dots \otimes a_k\otimes n)^\star)=g_2(m^\star\otimes a_1\otimes \dots \otimes a_k\otimes n)= \\ = g_1[m^\star\otimes a_1\otimes\dots\otimes a_k](n)=(g_1[m\otimes a_1\otimes\dots\otimes a_k])^\star(n)
  \\
   = (g_1[m\otimes a_1\otimes\dots\otimes a_k](n))^\star=(g_2(m\otimes a_1\otimes \dots\otimes a_k\otimes n))^\star. \nonumber
\end{multline}

The rest of the proof is standard and follows the steps of Theorem 2.75 \cite{Rotman09} or Proposition 2.6.3 \cite{Weibel94}.
\end{proof}



For an $A$-bimodule $M$, let us define $(-)\widetilde{\boxtimes}_\infty M$ as the covariant functor
\begin{displaymath}
\left.\begin{array}{cccc} \overline{A\mhyphen\ibimod} & \xrightarrow{(-)\widetilde{\boxtimes}_\infty M} & \overline{A\mhyphen\ibimod} \\ B & \rightsquigarrow & B\widetilde{\boxtimes}_\infty M\end{array}\right. .
\end{displaymath}

This functor sends a map $B_1\stackrel{f}{\longrightarrow} B_2$ to $B_1\widetilde{\boxtimes}_\infty M\xrightarrow{f\widetilde{\boxtimes}_\infty \id_M} B_2\widetilde{\boxtimes}_\infty M$.
\\\\
The functor $(-)\widetilde{\boxtimes}_\infty M$ is involutive: let us consider an involutive map $f:B_1\to B_2$ and its image under the tensor product functor, $g=f\widetilde{\boxtimes}_\infty \id_M$. Hence:
\begin{equation*}
g((b,a)^\star)=g(b^\star,a)=(f(b^\star),a)=(f(b),a)^\star=(g(b,a))^\star.
\end{equation*}

Given an involutive $A_\infty$-algebra $A$, we say that an involutive $A$-bimodule $F$ is \index{Flat $A_\infty$-bimodule} \textit{flat} if the tensor product functor $(-)\widetilde{\boxtimes}_\infty F:\overline{A\mhyphen \ibimod}\to \overline{A\mhyphen \ibimod}$ is exact, that is: it takes quasi-isomorphisms to quasi-isomorphisms. From now on, we will assume that all the involutive $A$-bimodules are flat.

\begin{lemma}
If $P$ and $Q$ are homotopy equivalent as involutive $A_\infty$-bimodules then, for every involutive $A_\infty$-bimodule $M$, the following quasi-isomorphism in the category of involutive $A_\infty$-bimodules holds: 
$$P\widetilde{\boxtimes}_\infty M\simeq Q\widetilde{\boxtimes}_\infty M.$$
\end{lemma}

\begin{proof}
Let $f:P\leftrightarrows Q:g$ be a homotopy equivalence. It is clear that 
$$h\sim k\Rightarrow h\widetilde{\boxtimes}_\infty\id_M\sim k\widetilde{\boxtimes}_\infty\id_M.$$
Therefore, we have:
\begin{displaymath}
\left.\begin{array}{ccccc}P\widetilde{\boxtimes}_\infty M & \to & Q\widetilde{\boxtimes}_\infty M & \to & P\widetilde{\boxtimes}_\infty M \\p\widetilde{\boxtimes}a & \mapsto & f(p)\widetilde{\boxtimes}a & \mapsto & g(f(p))\widetilde{\boxtimes}a\end{array}\right.
\end{displaymath}
and 
\begin{displaymath}
\left.\begin{array}{ccccc}Q\widetilde{\boxtimes}_\infty M & \to & P\widetilde{\boxtimes}_\infty M & \to & Q\widetilde{\boxtimes}_\infty M \\q\widetilde{\boxtimes}a & \mapsto & g(q)\widetilde{\boxtimes}a & \mapsto & f(g(q))\widetilde{\boxtimes}a\end{array}\right.
\end{displaymath}
\\
the result follows since $f\circ g\sim \id_Q$ and $g\circ f\sim \id_P$.
\end{proof}

\begin{lemma} \label{homoinv}
Let $A$ be an involutive $A_\infty$-algebra. If $P$ and $Q$ are homotopy equivalent as involutive $A$-bimodules then, for every involutive $A$-bimodule $M$, the following quasi-isomorphism holds:
$$\homo_{\overline{A\mhyphen\ibimod}}(P,M)\simeq\homo_{\overline{A\mhyphen\ibimod}}(Q,M).$$
\end{lemma}

\begin{proof}
Consider $f:P\to Q$ a homotopy equivalence and let $g:Q\to P$ be its homotopy inverse. If $[-,-]$ denotes the homotopy classes of morphisms, then both $f$ and $g$ induce the following maps:

\begin{displaymath}
\begin{array}{cccc}f_\star: & [P,M] & \rightarrow & [Q,M] \\ & \alpha & \mapsto & \alpha\circ g\end{array}
\end{displaymath}
\begin{displaymath}
\begin{array}{cccc}g_\star: & [Q,M] & \rightarrow & [P,M] \\ & \beta & \mapsto & \beta\circ f\end{array}
\end{displaymath}

Now we have:
\begin{displaymath}
f_\star\circ g_\star\circ\beta=f_\star\circ \beta\circ f=\beta\circ g\circ f\sim\beta;
\end{displaymath}
\begin{displaymath}
g_\star\circ f_\star\circ\alpha=g_\star\circ \alpha\circ g=\alpha\circ f\circ g\sim\alpha. \qedhere
\end{displaymath}
\end{proof}







\section{Involutive Hochschild homology and cohomology} 

\subsection{Hochschild homology for involutive $A_\infty$-algebras} 

We define the \index{Involutive Hochschild chains} \textit{involutive Hochschild chain complex} of an involutive $A_\infty$-algebra $A$ with coefficients in an $A$-bimodule $M$ as follows:
$$C^{\text{inv}}_\bullet(M,A)=M\widetilde{\boxtimes}_\infty \ba(A).$$

The differential is the same given in Section 7.2.4 \cite{KoSo09}.
The involutive Hochschild homology of $A$ with coefficients in $M$ is 
$$\hh_n(M,A)=\h C^{\text{inv}}_n(M,A).$$ 

\begin{lemma}
For an involutive $A_\infty$-algebra $A$ and a flat $A$-bimodule $M$, the following quasi-isomorphism holds: 
$$C^{\text{inv}}_{\bullet}(M,A)\simeq M\widetilde{\boxtimes}_\infty A.$$
\end{lemma}

\begin{proof}
The result follows from:
\begin{equation*}
M\widetilde{\boxtimes}_\infty A \simeq M\widetilde{\boxtimes}_\infty \ba(A)=C^{\text{inv}}_\bullet(M,A).
\end{equation*}
Observe that we are using that $M$ is flat and that there is a quasi-isomorphism between $\ba(A)$ and $A$ (Proposition 2, Section 2.3.1 \cite{Ferrario12}). 
\end{proof}



\subsection{Hochschild cohomology for involutive $A_\infty$-algebras} 

The \index{Involutive Hochschild cochains} \textit{involutive Hochschild cochain complex} of an involutive $A_\infty$-algebra $A$ with coefficients on an $A$-bimodule $M$ is defined as the $\K$-vector space
\begin{displaymath}
C_{\text{inv}}^\bullet(A,M):=\homo_{\overline{A\mhyphen\ibimod}}(\ba(A), M),
\end{displaymath} 
with the differential defined in section 7.1 of \cite{KoSo09}.

\begin{proposition}
For an involutive $A_\infty$-algebra $A$ and an $A$-bimodule $M$, we have the following quasi-isomorphism: $C_{\text{inv}}^{\bullet}(A,M)\simeq\homo_{\overline{A\mhyphen\ibimod}}(A,M)$.
\end{proposition}


\begin{proof}
The result follows from:
\begin{multline*}
C_{\text{inv}}^{\bullet}(A,M)=\homo_{\overline{A\mhyphen\ibimod}}(\ba(A), M):= \\
\homo^n_{\ivect}(\ba(A)\otimes_{\K}S\ba(A)\otimes_{\K}\ba(A), \ba(A)\otimes_{\K}SM\otimes_{\K}\ba(A)) \stackrel{(!)}{\simeq} \\
\homo^n_{\ivect}(\ba(A)\otimes_{\K}SA\otimes_{\K}\ba(A), \ba(A)\otimes_{\K}SM\otimes_{\K}\ba(A))=: \\
\homo_{\overline{A\mhyphen\ibimod}}(A, M).
\end{multline*}
Here (!) points out the fact that $S\ba(A)$ is a projective resolution of $SA$ in $\ivect$ and hence we have the quasi-isomorphism $S\ba(A)\simeq SA$. Observe that $S\ba(A)$ is projective in $\ivect$, therefore the involved functors in the proof are exact and preserve quasi-isomorphisms. \qedhere

\end{proof}





\bibliographystyle{amsalpha}
\bibliography{/Users/Ramses/Documents/Universitat/Paper_Algebras/Bibliografia}

\end{document}